\documentclass[english,11pt]{article}
\usepackage{graphicx}
\usepackage{amsmath}
\usepackage{amssymb}
\usepackage{pst-node}
\usepackage{color}
\usepackage{amsthm}
\usepackage{graphicx}
\usepackage{epstopdf}
\usepackage{tikz}
\usepackage{pgfplots}
\usepackage{pgflibraryarrows}
\usepackage{pgflibrarysnakes}
\usepackage{ifpdf}
\usepackage{setspace}
\RequirePackage{marginnote,hyperref}

\newtheorem{theorem}{Theorem}
\newtheorem{corollary}[theorem]{Corollary}

\newtheorem{observation}{Observation}

\newtheorem{lemma}{Lemma}

\newcommand{\Z}{\mbox{$\mathbb Z$}}
\newcommand{\M}{\mbox{$\mathcal M$}}

\begin{document}
\title{Flow Extensions and Group Connectivity with Applications }

\author{ Jiaao Li\thanks{Research  supported by  National Natural Science Foundation of China (No. 11901318), Natural Science Foundation of Tianjin (No. 19JCQNJC14100) and  the Fundamental Research Funds for the Central Universities, Nankai University (No. 63191425)}\\\small School of Mathematical Sciences and LPMC, Nankai University, Tianjin 300071,  China\\\small Email: lijiaao@nankai.edu.cn
}

\date{}
\maketitle

\begin{abstract}
We study the flow extension of graphs, i.e., pre-assigning a partial flow on the edges incident to a given vertex and aiming to extend to the entire graph. This is closely related to Tutte's $3$-flow conjecture(1972) that every $4$-edge-connected graph admits a nowhere-zero $3$-flow and a $\Z_3$-group connectivity conjecture(3GCC) of Jaeger, Linial, Payan, and Tarsi(1992) that every $5$-edge-connected graph $G$ is $\Z_3$-connected. Our main results show that these conjectures are equivalent to their natural flow extension versions and present some applications.
The $3$-flow case gives an alternative proof of Kochol's result(2001) that Tutte's $3$-flow conjecture is equivalent to its restriction on $5$-edge-connected graphs and is implied by the 3GCC. It also shows a new fact that Gr{\"o}tzsch's  theorem (that triangle-free planar graphs are $3$-colorable) is equivalent to its seemly weaker girth five case that  planar graphs of grith $5$ are $3$-colorable.
Our methods  allow to verify 3GCC for graphs with crossing number one, which is in fact reduced to the planar case proved by Richter, Thomassen and Younger(2017). Other equivalent versions of 3GCC and related partial results are obtained as well. 
\\[2mm]
\textbf{Keywords:} nowhere-zero flow; $3$-flow conjecture; flow extension; group connectivity
\\[2mm] \textbf{AMS Subject Classification (2010):} 05C15, 05C21, 05C40
\end{abstract}

\section{Introduction}

We consider finite graphs without loops, but permitting parallel edges.  {A vertex of degree $k$ is called a {\em $k$-vertex}.}  An edge-cut of size $k$ is called a {\em $k$-cut} for convenience, and basically no vertex-cut would be involved in this paper. A graph is {\em essentially $k$-edge-connected} if for any $t<k$, every $t$-cut isolates a vertex. In  a graph $G$, a function $\beta : V(G) \rightarrow \mathbb{Z}_3$ is called a {\em boundary function} of $G$ if $\sum_{x\in V(G)}\beta(x)=0$ in $\Z_3$. Let $Z(G, \Z_3)$ be the set of all  boundary functions of $G$. We call an orientation $D$ of $G$ a {\em $\beta$-orientation} if it holds that $d^+_D(v)-d^-_D(v)= \beta(v)$ in $\mathbb{Z}_3$ for every vertex $v\in V(G)$.  The special case of $\beta$-orientation with $\beta(x)=0$ in $\mathbb{Z}_3$ for every vertex $x\in V(G)$ is known as a  {\em mod $3$-orientation} of $G$. It is well-known (cf.\cite{LTWZ13,Tutt66,Zhan97}) that  searching for mod $3$-orientations is equivalent to finding nowhere-zero $3$-flows in graphs. Tutte's $3$-Flow Conjecture (abbreviated as 3FC) in 1972(see \cite{BoMu08})  is as follows.

\begin{center}
{\bf $3$-Flow Conjecture (3FC):} {\em Every $4$-edge-connected graph admits a nowhere-zero $3$-flow.}
\end{center}

The 3FC restricted to planar graphs is the dual of Gr{\"o}tzsch's $3$-Coloring Theorem (3CT) that every triangle-free planar graph is $3$-colorable. Applying the famous coloring extension techniques, Thomassen \cite{Thom94,Thom95,Thom03} presented short proofs of  Gr{\"o}tzsch's 3CT and extended to its list version, as well as obtained his elegant $5$-list-coloring theorem \cite{Thom94list5}. Even before Thomassen's coloring extension proofs, Steinberg and Younger \cite{SY89} employed a flow extension method to confirm 3FC for planar  and projective planar graphs, that is to pre-assign certain flow value to edges incident a given vertex and then to extend it to the entire graph. 
Motivated by the results of Steinberg and Younger, we say that a graph $G$ is {\em $\M_3$-extendable at  $z\in V(G)$} if for any pre-orientation $D_{0}$ of $\partial_G(z)$ with $d_{D_{0}}^+(z) \equiv d_{D_{0}}^-(z) \pmod 3$,
$D_{0}$ can be extended to a mod $3$-orientation $D$ of $G$.

\medskip
Kochol \cite{Koch01} obtained some interesting equivalent versions of the 3FC.
\begin{theorem}\label{kochol}(Kochol \cite{Koch01}) The following  are equivalent.
\\
(i) {\em(3FC)} Every $4$-edge-connected graph admits a nowhere-zero $3$-flow.
\\
(ii) Every $5$-edge-connected graph admits a nowhere-zero $3$-flow.
\\
(iii) Every $5$-edge-connected  graph is $\M_3$-extendable at every $5$-vertex.
\end{theorem}

A graph $G$ is called {\em $\mathbb{Z}_3$-connected}
if, for every $\beta\in {Z}(G, \mathbb{Z}_{3})$, there exists a $\beta$-orientation in $G$ (i.e., an orientation
$D$ such that $d^+_D(x)-d^-_D(x)\equiv \beta(x)  \pmod {3},\forall x\in V(G)$).   This group connectivity concept was  introduced by Jaeger, Linial, Payan, and Tarsi \cite{JLPT92} as a nonhomogeneous generalization of Tutte's nowhere-zero flow theory \cite{Tutt54}.  Jaeger et al. \cite{JLPT92} posed the following $\Z_3$-Group Connectivity Conjecture, abbreviated as 3GCC.

%

\begin{center}
\noindent{\bf$\Z_3$-Group Connectivity Conjecture (3GCC):}{\em Every $5$-edge-connected graph is $\Z_3$-connected.}
\end{center}


The main purpose of this paper is to study some natural flow extension versions of 3FC and 3GCC, with some additional applications. In particular, using a unified approach,  we provide a new proof of Theorem \ref{kochol} (different from Kochol's $2$-sum method \cite{Koch01}), and  prove that some seemly stronger versions of 3GCC are actually equivalent to the original version, as shown in Theorem \ref{THM: TFAE} below. Furthermore, as a byproduct of the new proof of Kochol's Theorem \ref{kochol}, it also indicates that those statements are equivalent within planar graphs, which implies that, by duality,  Gr{\"o}tzsch's 3CT is exactly equivalent to its restriction on grith $5$ case. This  interesting fact seems not known before (since Kochol's arguments \cite{Koch01} need to construct nonplanar graphs).

Similar as the $\M_3$-extendability on mod $3$-orientations, there is an analogous pre-orientation  extension concept for $\Z_3$-group connectivity. This technique is notably one of the key ideas in the proof of Weak $3$-Flow Conjecture by Thomassen \cite{Thom12}, and subsequently improvement by Lov\'asz, Thomassen, Wu and Zhang \cite{LTWZ13}.   A graph is called {\em $\Z_3$-extendable at  $x$},
if  for any $\beta\in Z(G, \Z_3)$ and any pre-orientation $D_{x}$ of $\partial_G({x})$ with $d_{D_{x}}^+(x)-d_{D_{x}}^-(x)\equiv\beta (x)\pmod 3$,
$D_{x}$ can be extended to a $\beta$-orientation $D$ of $G$. A graph is {\em $\Z_3$-reduced} if it contains no $\Z_3$-connected subgraph of order at least two. We show the following statements are all equivalent to 3GCC, some of which have been appeared in \cite{HLL18} and  shown to imply the 3GCC.

\begin{theorem}\label{THM: TFAE} The following  are equivalent.
\\
(a) {\em(3GCC)} Every $5$-edge-connected graph is $\Z_3$-connected.
\\
(b-i) Every $5$-edge-connected graph is $\Z_3$-extendable at every $5$-vertex.
\\
(b-ii) Every $5$-edge-connected essentially $6$-edge-connected  graph is $\Z_3$-extendable at every  $5$-vertex.
\\
(c) Every $\Z_3$-reduced graph has minimum degree at most $4$.
\\
(d) Every $4$-edge-connected graph with at most five $4$-cuts is $\Z_3$-connected.
\end{theorem}
In particular, Theorem \ref{THM: TFAE}, using equivalent statement (c), provides  another alternative proof (different from Theorem \ref{kochol}) of the fact that the validity of 3GCC implies 3FC. To see this, notice that the minimal counterexample $G$ of 3FC is $5$-regular by Mader's splitting lemma \cite{Made78} (Lemma \ref{maderlem} below). Observe also that, if $H$ is a $\Z_3$-connected subgraph of $G$, then a mod $3$-orientation of $G/H$ can be easily extended to $G$ (cf.\cite{HLL18,JLPT92,LTWZ13,Zhan97}), and so the minimal counterexample $G$ must be $\Z_3$-reduced. Thus $G$ is a $5$-regular $\Z_3$-reduced graph, a contradiction to Theorem \ref{THM: TFAE} (c).

~

Restricted to planar graphs, applying the powerful flow extension techniques, a recent result of  Richter, Thomassen and Younger \cite{RTY16} shows  3GCC and its flow extension version(Theorem \ref{THM: TFAE}(b-i))  hold for planar graphs. The techniques in proving Theorems \ref{kochol} and \ref{THM: TFAE} allow us to obtain more equivalent statements of the Richter-Thomassen-Younger result, and to extend it to graphs with crossing number one.
\begin{theorem}\label{THM:pla}  Each of the following holds.\\
(i)(\cite{LL06,RTY16}) Every $5$-edge-connected planar graph is $\Z_3$-connected.
\\
(ii) (\cite{RTY16}) Every $5$-edge-connected planar graph is $\Z_3$-extendable at every $5$-vertex.
\\
(iii) Every $\Z_3$-reduced planar graph has minimum degree at most $4$.
\\
(iv) Every $5$-edge-connected graph with crossing number at most one is $\Z_3$-connected.
\end{theorem}

~

 For general graphs, we  summarize some previous approach on each of the above statements of Theorem \ref{THM: TFAE} from \cite{LTWZ13, HLL18}, and also provide new partial results for Theorem \ref{THM: TFAE}(d).
\begin{theorem}\label{THM: ETFH}  Each of the following holds.
\\
(a) {\em(\cite{LTWZ13})} Every $6$-edge-connected graph is $\Z_3$-connected.
\\
(b-i) {\em(\cite{LTWZ13})} Every $6$-edge-connected graph is $\Z_3$-extendable at every vertex of degree at most $7$.
\\
(b-ii) {\em(\cite{HLL18})} Every $5$-edge-connected essentially $23$-edge-connected  graph is $\Z_3$-extendable at every $5$-vertex.
\\
(c) (\cite{HLL18}) Every $\Z_3$-reduced graph has minimum degree at most $5$.\\
(d-i) Every $4$-edge-connected graph with at most five $4$-cuts and without $5$-cuts is $\Z_3$-connected.\\
(d-ii) Every $5$-edge-connected graph with at most seven $5$-cuts is $\Z_3$-connected.
\end{theorem}

Note that Jaeger et al. \cite{JLPT92} constructed a $4$-edge-connected non-$\Z_3$-connected graph with fifteen $4$-cuts and without $5$-cuts. This  indicates that Theorem \ref{THM: ETFH}(d-i) is almost tight.

In the next section, we first present some preliminaries, and then prove Theorems \ref{kochol}, \ref{THM: TFAE} and \ref{THM:pla}. The proof of Theorem \ref{THM: ETFH} (d-i)(d-ii) will be completed in Section 3.

\section{Flow Extensions}


\subsection{Preliminaries}
Before proceeding we introduce a few more notation. For a vertex subset $A\subset V(G)$, we use $\partial_G(A)$ to denote the set of edges with one end in $A$ and the other in $A^c$, where $A^c=V(G)\setminus A$ is the complement of $A$. Let $d_G(A)=|\partial_G(A)|$ be the number of edges between $A$ and $A^c$. When $A=\{x\}$, we shall use $\partial_G(x)$ for $\partial_G(\{x\})$ and $d_G(x)$ for $d_G(\{x\})$, respectively. Sometimes the subscripts may be omitted for convenience if the graph $G$ is understood from context.

In a graph $G$, a $k$-cut $\partial(A)$ is called {\em a $k$-critical-cut} with respected to $A$ if $d(A)\le k$ and for any $B\subsetneq A$, $d(B)>k$; we also say that $A$ is a {\em $k$-critical-set}. The following observation follows easily from the definition.

\begin{observation}\label{OB: kedgeconnected}
  Let $G$ be a $k$-edge-connected graph with exactly  $q$  $k$-cuts. Denote $A_1, A_2, \dots, A_t$ to be all distinct $k$-critical-set $A$ such that $\partial(A)$ is a $k$-critical-cut. Then each of the following holds.\\
  (i) $A_i\cap A_j= \emptyset$ for any $i\neq j$.
  \\
  (ii) If $q=1$, then $t=2$ and $A_2=V(G)\setminus A_1$.\\
  (iii) If $q\ge 2$, then $\partial(A_i)\neq \partial(A_j)$ for any $i\neq j$.  Hence $t\le q$.\\
  (iv)  Construct a graph $G'$ from $G$ by adding a new vertex $x$ and connecting $x$ and $A_i$ with a new edge for each $i=1, \ldots, t$. Then all the edge-cuts other than $\partial_{G'}(x)$ in $G'$ have size at least $k+1$.
\end{observation}

Let $G$ be a graph with a $5$-vertex $x\in V(G)$. In a mod $3$-orientation $D$ of $G$,  the  edges in $\partial(x)$ at $x$  is oriented either as $4$ ingoing and $1$ outgoing, or as $1$ ingoing and $4$ outgoing. So we call such an edge in $\partial(x)$  a {\em minor-edge} at $x$ if its orientation is different from other edges in $\partial(x)$.

A major step of our arguments relies on the following property of flows on the graph $W$ depicted in  Figure \ref{FIGW}. Formally, $W$ denotes the graph with vertex set $V(W)=\{v_0,v_1,\ldots,v_5\}$ and edge multiset  $$E(W)=\{v_5v_1,v_5v_1\}\bigcup_{i=1}^4\{v_iv_{i+1}, v_iv_{i+1}\}\bigcup_{i=1}^5\{v_0v_i\}.$$

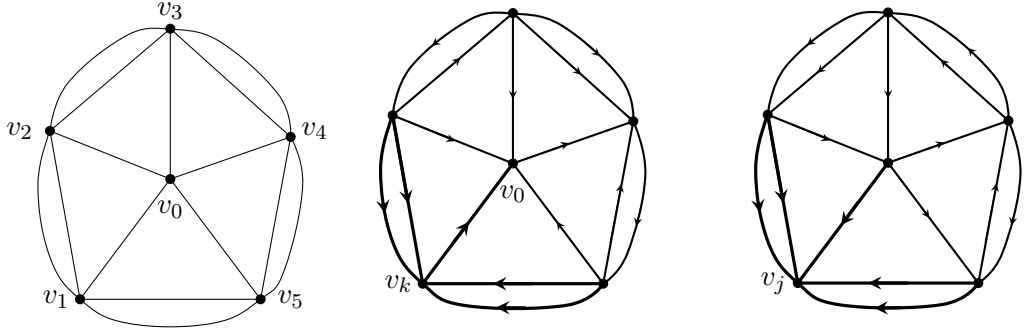
\begin{figure}
\centering
\begin{minipage}[h]{0.3\textwidth}
\begin{tikzpicture}[scale=0.8]
\tikzstyle{mynodestyle} = [draw,shape=circle,outer sep=0,inner sep=1.2,minimum size=2,fill=black]

\node [mynodestyle] (v2) at (0,-0.5) {};
\node (va0) at (0,-1) {$v_0$};
\node [mynodestyle] (v1) at (0,2) {};
\node (va3) at (0,2.3) {$v_3$};
\node [mynodestyle] (v3) at (-2,0.3) {};
\node (va2) at (-2.5,.3) {$v_2$};
\node [mynodestyle] (v6) at (2,0.2) {};
\node (va4) at (2.4,.3) {$v_4$};
\node [mynodestyle] (v4) at (-1.5,-2.5) {};
\node (va1) at (-1.9,-2.5) {$v_1$};
\node [mynodestyle] (v5) at (1.5,-2.5) {};
\node (va5) at (2,-2.5) {$v_5$};
\draw  (v1) edge (v2);
\draw  (v3) edge (v2);
\draw  (v2) edge (v4);
\draw  (v2) edge (v5);
\draw  (v2) edge (v6);

\draw  (v1) edge (v6);
\draw  (v6) edge (v5);
\draw  (v4) edge (v5);
\draw  (v3) edge (v4);
\draw  (v3) edge (v1);
\draw  plot[smooth, tension=.7] coordinates {(v3) (-1.7,1.1) (-0.7,1.9) (v1) (0,2) (0.8,1.8) (1.8,0.9) (v6)};
\draw  plot[smooth, tension=.7] coordinates {(v6) (2.2,-0.6) (1.9,-2) (v5) (0.7,-2.9) (-0.7,-2.9) (v4) (-2,-1.8) (-2.2,-0.6) (v3)};
\end{tikzpicture}
\end{minipage}
\begin{minipage}[h]{0.3\textwidth}
\begin{tikzpicture}[scale=0.8]
\tikzstyle{mynodestyle} = [draw,shape=circle,outer sep=0,inner sep=1.2,minimum size=2,fill=black]

\node [mynodestyle] (v2) at (0,-0.5) {};
\node [mynodestyle] (v1) at (0,2) {};
\node [mynodestyle] (v3) at (-2,0.3) {};
\node [mynodestyle] (v6) at (2,0.2) {};
\node [mynodestyle] (v4) at (-1.5,-2.5) {};
\node [mynodestyle] (v5) at (1.5,-2.5) {};
\node (va1) at (-1.9,-2.5) {$v_k$};
\node (va0) at (0,-1) {$v_0$};
\draw [ thick](v1) edge (v2);
\draw [ thick]  (v3) edge (v2);
\draw [very thick]  (v2) edge (v4);
\draw [ thick] (v2) edge (v5);
\draw [ thick] (v2) edge (v6);

\draw [ thick] (v1) edge (v6);
\draw[ thick]  (v6) edge (v5);
\draw[very thick]  (v4) edge (v5);
\draw [very thick] (v3) edge (v4);
\draw[ thick]  (v3) edge (v1);
\draw [ thick] plot[smooth, tension=.7] coordinates {(v3) (-1.7,1.1) (-0.7,1.9) (v1) (0,2) (0.8,1.8) (1.8,0.9) (v6)};
\draw[very thick]  plot[smooth, tension=.7] coordinates {(v5) (0.8925,-2.8587) (-0.791,-2.8637) (v4) (-2,-1.8) (-2.2,-0.6) (v3)};

\node (v7) at (-1.7108,-1.3392) {};
\draw [very thick,-stealth] (v3) edge (v7);
\node (v8) at (-2.1822,-0.9903) {};
\node (v9) at (-2.1,-1.5) {};
\draw [very thick,-stealth] (v8) edge (v9);
\node (v10) at (-1.1321,-2.0033) {};
\node (v11) at (-0.6,-1.3) {};
\draw [very thick,-stealth] (v10) edge (v11);
\node (v12) at (0.5,-2.5) {};
\node (v13) at (-0.5,-2.5) {};
\draw [very thick,-stealth] (v12) edge (v13);
\node (v14) at (0.39,-2.9048) {};
\node (v15) at (-0.4494,-2.8946) {};
\draw [very thick,-stealth] (v14) edge (v15);
\draw[ thick]  plot[smooth, tension=.7] coordinates {(v6) (2.2,-0.7) (1.9,-2) (v5)};

\node (v17) at (0,0.4) {};
\node (v19) at (-0.8,-0.2) {};
\node (v16) at (0,2) {};
\node (v18) at (-2,0.3) {};
\node at (-1.5,-2.5) {};
\node (v20) at (1.5,-2.5) {};
\node at (2,0.2) {};
\draw [-stealth] (v16) edge (v17);
\draw [-stealth] (v18) edge (v19);
\node (v21) at (0.6,-1.3) {};
\draw [-stealth] (v20) edge (v21);
\node (v22) at (1.1,-0.1) {};
\draw [-stealth] (v2) edge (v22);
\node (v23) at (1.5806,-1.9931) {};
\node (v24) at (1.8329,-0.672) {};
\node (v25) at (2.2,-0.9) {};
\node (v26) at (2.0085,-1.7127) {};
\draw [-stealth] (v23) edge (v24);
\draw [-stealth] (v25) edge (v26);
\node (v27) at (0.5203,1.5289) {};
\node (v28) at (1.3056,0.8211) {};
\draw [-stealth] (v27) edge (v28);
\node (v29) at (1.0433,1.6319) {};
\node (v30) at (1.6186,1.1384) {};
\draw [-stealth] (v29) edge (v30);
\node (v31) at (-1.4,0.8) {};
\node (v32) at (-0.7,1.4) {};
\draw [-stealth] (v31) edge (v32);
\node (v33) at (-0.932,1.7786) {};
\node (v34) at (-1.5171,1.3231) {};
\draw [-stealth] (v33) edge (v34);
\end{tikzpicture}
\end{minipage}
\begin{minipage}[h]{0.3\textwidth}
\begin{tikzpicture}[scale=0.8]
\tikzstyle{mynodestyle} = [draw,shape=circle,outer sep=0,inner sep=1.2,minimum size=2,fill=black]

\node (va1) at (-1.9,-2.5) {$v_j$};
\node [mynodestyle] (v2) at (0,-0.5) {};
\node [mynodestyle] (v1) at (0,2) {};
\node [mynodestyle] (v3) at (-2,0.3) {};
\node [mynodestyle] (v6) at (2,0.2) {};
\node [mynodestyle] (v4) at (-1.5,-2.5) {};
\node [mynodestyle] (v5) at (1.5,-2.5) {};
\draw [ thick](v1) edge (v2);
\draw [ thick]  (v3) edge (v2);
\draw [very thick]  (v2) edge (v4);
\draw [ thick] (v2) edge (v5);
\draw [ thick] (v2) edge (v6);

\draw [ thick] (v1) edge (v6);
\draw[ thick]  (v6) edge (v5);
\draw[very thick]  (v4) edge (v5);
\draw [very thick] (v3) edge (v4);
\draw[ thick]  (v3) edge (v1);
\draw [ thick] plot[smooth, tension=.7] coordinates {(v3) (-1.7,1.1) (-0.7,1.9) (v1) (0,2) (0.8,1.8) (1.8,0.9) (v6)};
\draw[very thick]  plot[smooth, tension=.7] coordinates {(v5) (0.8925,-2.8587) (-0.791,-2.8637) (v4) (-2,-1.8) (-2.2,-0.6) (v3)};

\node (v7) at (-1.7108,-1.3392) {};
\draw [very thick,-stealth] (v3) edge (v7);
\node (v8) at (-2.1822,-0.9903) {};
\node (v9) at (-2.1,-1.5) {};
\draw [very thick,-stealth] (v8) edge (v9);

\node (v12) at (0.5,-2.5) {};
\node (v13) at (-0.5,-2.5) {};
\draw [very thick,-stealth] (v12) edge (v13);
\node (v14) at (0.39,-2.9048) {};
\node (v15) at (-0.3947,-2.9207) {};
\draw [very thick,-stealth] (v14) edge (v15);
\draw[ thick]  plot[smooth, tension=.7] coordinates {(v6) (2.2,-0.7) (1.9,-2) (v5)};

\node (v17) at (0,0.4) {};
\node (v19) at (-0.8,-0.2) {};
\node (v16) at (0,2) {};
\node (v18) at (-2,0.3) {};
\node at (-1.5,-2.5) {};

\node at (2,0.2) {};
\draw [-stealth] (v16) edge (v17);
\draw [-stealth] (v18) edge (v19);

\node (v22) at (1.1,-0.1) {};
\draw [-stealth] (v2) edge (v22);
\node (v23) at (1.5806,-1.9931) {};
\node (v24) at (1.8329,-0.672) {};
\node (v25) at (2.2,-0.9) {};
\node (v26) at (2.0085,-1.7127) {};
\draw [-stealth] (v23) edge (v24);
\draw [-stealth] (v25) edge (v26);

\node (v33) at (-0.932,1.7786) {};
\node (v34) at (-1.5171,1.3231) {};
\draw [-stealth] (v33) edge (v34);
\node (v10) at (-0.3713,-0.9862) {};
\node (v11) at (-0.9082,-1.7091) {};
\draw [very thick,-stealth] (v10) edge (v11);
\node (v31) at (-0.6292,1.4658) {};
\node (v32) at (-1.3479,0.8486) {};
\draw [-stealth] (v31) edge (v32);
\node (v27) at (1.269,0.8571) {};
\node (v28) at (0.7025,1.3686) {};
\draw [-stealth] (v27) edge (v28);
\node (v29) at (1.6495,1.0896) {};
\node (v30) at (1.1548,1.5758) {};
\draw [-stealth] (v29) edge (v30);
\node (v20) at (0.8166,-1.5866) {};
\draw [-stealth] (v2) edge (v20);
\end{tikzpicture}
\end{minipage}
\caption{The graph $W$ and its orientations.}\label{FIGW}
\end{figure}

\begin{lemma}\label{OB: W}
(i) For any mod $3$-orientation $D$ of $W$, there exists a vertex $v_k$ with $1\le k\le 5$ such that $v_0v_k$ is the minor-edge at $v_k$.

   (ii) Let $\beta\in Z(W,\Z_3)$ be a boundary function such that $\beta(v_i)=1$ in $\mathbb{Z}_3$ for each $i=0,1,\ldots,5$. Then for any $\beta$-orientation $D$ of $W$, there exists a vertex $v_{j}\in V(W)$ such that  $d_D^+(v_{j})=0$ and $d_D^-(v_{j})=5$.
\end{lemma}
\begin{proof}
  (i) Suppose to the contrary that, in a mod $3$-orientation $D$ of $G$ each edge $v_0v_k$ is not the minor-edge at $v_k$ for $k=1,2,\ldots,5$. We count the {\em deficiency} $d_D^+(v)-d_D^-(v)$ at each vertex $v\in V(W)$.  By symmetry, we may assume that under orientation $D$ the edges in $\partial(v_0)$ at vertex $v_0$ are oriented as $4$ ingoing and $1$ outgoing (with deficiency $-3$). As each  $v_0v_k$ is not the minor-edge at $v_k$ for $k=1,2,\ldots,5$, it holds that  four of $\{v_1,v_2,\ldots,v_5\}$ are received orientations as $1$ ingoing and $4$ outgoing (with deficiency $3$), and the other one is opposite as $4$ ingoing and $1$ outgoing (with deficiency $-3$). So the deficiency at all the vertices are four $3$'s and two $-3$'s. This is a contradiction to the fact that $\sum_{v\in V(W)}(d_D^+(v)-d_D^-(v))=0$.

  (ii) The proof is similar to (i) by counting deficiency at each vertex. Let $D$ be a $\beta$-orientation of $W$. Then for each vertex $v\in V(W)$, $d_D^+(v)-d_D^-(v)\equiv \beta(v)\equiv 1\pmod3$, and so the deficiency $d_D^+(v)-d_D^-(v)\in\{1,-5\}$. Since $\sum_{i=0}^5 (d_D^+(v_i)-d_D^-(v_i))=0$, there exists a vertex $v_{j}$ with $0\le j\le 5$ such that $d_D^+(v)-d_D^-(v)=-5$ as desired.
  \end{proof}

We also need the following lemma about $\Z_3$-extendability in \cite{HLL18}.
\begin{lemma}\label{extendingiff}{\em(\cite{HLL18})} Let $G$ be a graph with $x \in V(G)$.
 Then $G$ is $\Z_3$-extendable at $x$ if and only if $G-x$ is $\Z_3$-connected.
\end{lemma}

For a graph $G$ with $uz, vz\in E(G)$, a {\em splitting} at $z$ is an operation to delete edges $uz,vz$ and add a new edge $uv$. If $z$ is an even vertex of $G$, a {\em complete splitting} at $z$ is to apply splitting operations on all the edges of $\partial_G(z)$ in pairs and then delete the isolated vertex $z$ to obtain the resulting graph. The following Mader's splitting lemma shows that it is possible to preserve the edge connectivity after splitting operations.
\begin{lemma}\label{maderlem}{\em(Mader \cite{Made78})}
  Let $G$ be a $k$-edge-connected graph with a $t$-vertex $z\in V(G)$. If $t\ge k+2$, then there exists a splitting at $z$ such that the resulting graph is $k$-edge-connected. If $t$ is even, then there exists a complete splitting at $z$ such that the resulting graph is $k$-edge-connected.
\end{lemma}

\subsection{Proofs of Theorems \ref{kochol}, \ref{THM: TFAE} and \ref{THM:pla}}

In this subsection, we present the proofs of Theorems \ref{kochol}, \ref{THM: TFAE} and \ref{THM:pla} using a unified construction method through properties given in Lemma \ref{OB: W}.

~

\begin{figure}[ht]
\begin{minipage}[t]{0.4\textwidth}
  \begin{tikzpicture}[scale=0.5]

\tikzstyle{mynodestyle} = [draw,shape=circle,outer sep=0,inner sep=0.7,minimum size=2,fill=black]

\draw  (0,1) node (v4) {} ellipse (1.5 and 1.5);
\draw  (0,7) node (v5) {} ellipse (1.5 and 1.5);
\draw  (-4.5,3.5) node (v3) {} ellipse (1.5 and 1.5);
\draw  (4.5,3.5) ellipse (1.5 and 1.5);
\draw  (-3.5,-3) ellipse (1.5 and 1.5);
\draw  (3.5,-3) ellipse (1.5 and 1.5);

\node (g1) at (-4.2,-3.3) {\small$G^1$};
\node (g2) at (-5,3.8) {\small$G^2$};
\node (g3) at (0,7.5) {\small$G^3$};
\node (g4) at (5,3.8) {\small$G^4$};
\node (g5) at (4,-3) {\small$G^5$};
\node (g0) at (0.2,1.1) {\small$G^0$};

\node (x01) at (-4,-2.38) {\tiny$x_0^1$};
\node (x02) at (-3.9,4.2) {\tiny$x_0^2$};
\node (x03) at (1.1,6.9) {\tiny$x_0^3$};
\node (x04) at (5.2,2.8) {\tiny$x_0^4$};
\node (g5) at (2.9,-3.6) {\tiny$x_0^5$};

\node [mynodestyle] (v1) at (-4,-2) {};
\node [mynodestyle] (v19) at (-3.5,-2) {};
\node [mynodestyle] (v20) at (-3,-2.5) {};
\node [mynodestyle] (v22) at (-2.5,-3) {};
\node [mynodestyle] (v24) at (-2.5,-3.5) {};
\node [mynodestyle] (v6) at (2.5,-3.5) {};
\node [mynodestyle] (v25) at (2.5,-3) {};
\node [mynodestyle] (v23) at (3,-2.5) {};
\node [mynodestyle] (v26) at (3.5,-2) {};
\node [mynodestyle] (v28) at (4,-2) {};
\node [mynodestyle] (v8) at (5,2.5) {};
\node [mynodestyle] (v29) at (4.5,2.5) {};
\node [mynodestyle] (v27) at (4,3) {};
\node [mynodestyle] (v33) at (3.5,3.5) {};
\node [mynodestyle] (v31) at (3.5,4) {};

\node [mynodestyle] (v15) at (-1,6.5) {};

\node [mynodestyle] (v32) at (0,6) {};
\node [mynodestyle] (v30) at (0.5,6) {};
\node [mynodestyle] (v10) at (1,6.5) {};

\node [mynodestyle] (v18) at (-4.5,2.5) {};

\node [mynodestyle] (v14) at (-3.5,3.5) {};

\node [mynodestyle] (v12) at (-3.5,4) {};
\node [mynodestyle] (v2) at (-1,0.5) {};
\node [mynodestyle] (v7) at (0,0) {};
\node [mynodestyle] (v9) at (1,1) {};
\node [mynodestyle] (v11) at (0.5,2) {};
\node [mynodestyle] (v13) at (-0.5,1.5) {};
\draw  (v1) edge (v2);
\draw  (v6) edge (v7);
\draw  (v8) edge (v9);
\draw  (v10) edge (v11);
\draw  (v12) edge (v13);
\draw  (v14) edge (v15);
\node [mynodestyle] (v17) at (-3.5,3) {};
\node [mynodestyle] (v16) at (-0.5,6.5) {};
\draw  (v16) edge (v17);
\node [mynodestyle] (v21) at (-4,2.5) {};
\draw  (v18) edge (v19);
\draw  (v20) edge (v21);
\draw  (v22) edge (v23);
\draw  (v24) edge (v25);
\draw  (v26) edge (v27);
\draw  (v28) edge (v29);
\draw  (v30) edge (v31);
\draw  (v32) edge (v33);
\end{tikzpicture}

\end{minipage}
\begin{minipage}[t]{0.05\textwidth}
  ~
\end{minipage}
\begin{minipage}[t]{0.45\textwidth}
\begin{tikzpicture}[scale=0.5]

\tikzstyle{mynodestyle} = [draw,shape=circle,outer sep=0,inner sep=0.8,minimum size=2,fill=black]
\tikzstyle{myedgestyle} = [-triangle 60]

\draw  (0,1) node (v4) {} ellipse (1.5 and 1.5);
\draw  (0,7) node (v5) {} ellipse (1.5 and 1.5);
\draw  (-4.5,3.5) node (v3) {} ellipse (1.5 and 1.5);
\draw  (4.5,3.5) ellipse (1.5 and 1.5);
\draw  (-3.5,-3) ellipse (1.5 and 1.5);
\draw  (3.5,-3) ellipse (1.5 and 1.5);

\node [mynodestyle] (v1) at (-4,-2) {};
\node [mynodestyle] (v19) at (-3.5,-2) {};
\node [mynodestyle] (v20) at (-3,-2.5) {};
\node [mynodestyle] (v22) at (-2.5,-3) {};
\node [mynodestyle] (v24) at (-2.5,-3.5) {};
\node [mynodestyle] (v6) at (2.5,-3.5) {};
\node [mynodestyle] (v25) at (2.5,-3) {};
\node [mynodestyle] (v23) at (3,-2.5) {};
\node [mynodestyle] (v26) at (3.5,-2) {};
\node [mynodestyle] (v28) at (4,-2) {};
\node [mynodestyle] (v8) at (5,2.5) {};
\node [mynodestyle] (v29) at (4.5,2.5) {};
\node [mynodestyle] (v27) at (4,3) {};
\node [mynodestyle] (v33) at (3.5,3.5) {};
\node [mynodestyle] (v31) at (3.5,4) {};

\node [mynodestyle] (v15) at (-1,6.5) {};

\node [mynodestyle] (v32) at (0,6) {};
\node [mynodestyle] (v30) at (0.5,6) {};
\node [mynodestyle] (v10) at (1,6.5) {};

\node [mynodestyle] (v18) at (-4.5,2.5) {};

\node [mynodestyle] (v14) at (-3.5,3.5) {};

\node [mynodestyle] (v12) at (-3.5,4) {};
\node [mynodestyle] (v2) at (-1,0.5) {};
\node [mynodestyle] (v7) at (0,0) {};
\node [mynodestyle] (v9) at (1,1) {};
\node [mynodestyle] (v11) at (0.5,2) {};
\node [mynodestyle] (v13) at (-0.5,1.5) {};
\draw  (v1) edge (v2);
\draw  (v6) edge (v7);
\draw  (v8) edge (v9);
\draw  (v10) edge (v11);
\draw  (v12) edge (v13);
\draw  (v14) edge (v15);
\node [mynodestyle] (v17) at (-3.5,3) {};
\node [mynodestyle] (v16) at (-0.5,6.5) {};
\draw  (v16) edge (v17);
\node [mynodestyle] (v21) at (-4,2.5) {};
\draw  (v18) edge (v19);
\draw  (v20) edge (v21);
\draw  (v22) edge (v23);
\draw  (v24) edge (v25);
\draw  (v26) edge (v27);
\draw  (v28) edge (v29);
\draw  (v30) edge (v31);
\draw  (v32) edge (v33);

\draw [very thick] (v18) edge (v19);
\draw [very thick,-stealth] (v18) edge (-4,.25);

\draw [very thick] (v21) edge (v20);
\draw [very thick,-stealth] (v21) edge (-3.5,0);

\draw [very thick] (v1) edge (v2);
\draw [very thick,-stealth] (v1) edge (-2.5,-.75);

\draw [very thick] (v23) edge (v22);
\draw [very thick,-stealth] (v23) edge (0.25,-2.75);

\draw [very thick] (v25) edge (v24);
\draw [very thick,-stealth] (v25) edge (0,-3.25);
\end{tikzpicture}

\end{minipage}
\caption{The constructed graph $H$ and its orientation for proving Theorem \ref{kochol}.}\label{FIGW3flow}
\end{figure}
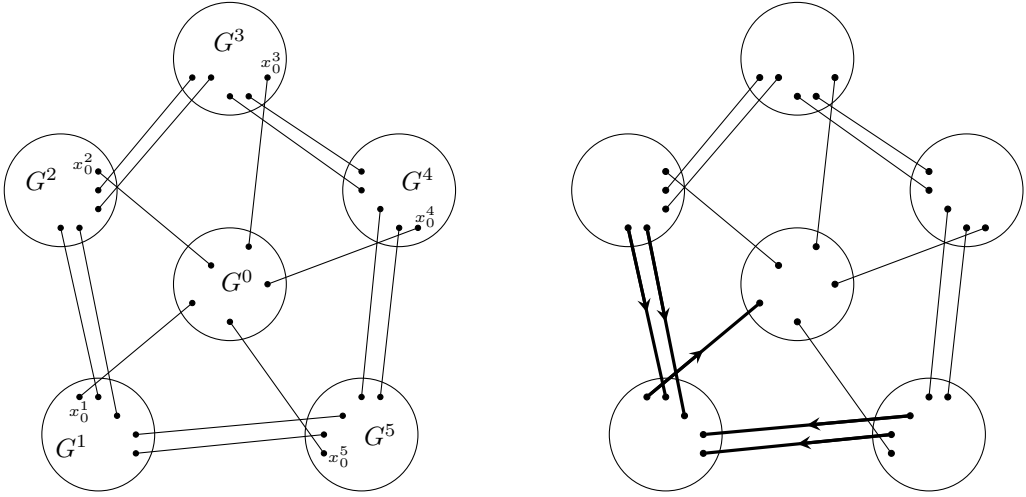

{\bf An Alternative Proof of Theorem \ref{kochol}:} Clearly, ``(i)$\Rightarrow$(ii)'' holds and a standard argument could show that ``(iii)$\Rightarrow$(i)''. We provide a proof of ``(iii)$\Rightarrow$(i)'' here for completeness, which is similar to Kochol's proof in \cite{Koch01}. Specifically, let $G$ be a counterexample of 3FC (statement (i)) with $|E(G)|+|V(G)|$ minimized. Then $G$ is $5$-regular by Lemma \ref{maderlem}. And $G$ must contain nontrivial $4$-cuts; otherwise $G$ is $5$-edge-connected, and so (i) follows by (iii). Among all nontrivial $4$-cuts of $G$, we select a $4$-cut $\partial(A)$ with $|A|$ as small as possible. Then $|V(G)|-1>|A|\ge 2$ and we have
\begin{equation}\label{EQ:A'ge5}
  \text{$d_G(A')=|\partial_G(A')|\ge 5$ for any $A'\subsetneq A$.}
\end{equation}  Contract $A$ to obtain a new graph $G_1=G/A$. Thus $G_1$ is $4$-edge-connected, and so admits a mod $3$-orientation $D_1$ by the minimality of $G$. Then we contract $A^c$ to obtain another new graph $G_2=G/A^c$, where $A^c$ is contracted to become a new vertex $x$. Pre-orient the edges in $\partial_{G_2}(x)$ the same as $\partial_{D_1}(A^c)$. Hence the edges in  $\partial_{G_2}(x)$ are oriented as two ingoing and two outgoing. Obtain a new graph $G_3$ from $G_2$ by replacing an ingoing edge at $x$ with two outgoing edges. Hence $x$ is a $5$-vertex now, and $G_3$ is $5$-edge-connected by (\ref{EQ:A'ge5}). Moreover, the pre-orientation at $x$ is still balanced mod $3$. By (iii), this pre-orientation can be extended to a mod $3$-orientation $D_3$ of $G_3$. Then, after deleting the edges of $\partial_{G_3}(x)$, the combination of $D_1$ and the rest of $D_3$ gives a mod $3$-orientation of $G$. Hence (iii) implies (i).

~

The major task remaining is to show that ``(ii)$\Rightarrow$(iii)''. The method below is principally different from Kochol's proof in  \cite{Koch01}. We hope this new method may shed some light on attacking 3FC and 3GCC.

Assume that statement (ii) holds that every $5$-edge-connected graph has a mod $3$-orientation. Suppose to the contrary that there is a $5$-edge-connected graph $G$ and a $5$-vertex $x\in V(G)$ with pre-orientation $D_x$ that is not $\M_3$-extendable to a mod $3$-orientation of $G$. Recall that $W$ denotes the graph depicted in Figure \ref{FIGW}. We construct a new graph $H$ by replacing each vertex of $W$ with a copy of $G-x$, where  each edge $v_0v_k$ ($1\le k\le 5$)  is corresponded to the minor-edge at $x$ of $D_x$ in that copy. More precisely, denote $\partial_G(x)=\{xx_0,xx_1,\ldots,xx_4\}$,  where $xx_0$ is the minor-edge in pre-orientation $D_x$. (Notice that we allow $x_i=x_j$ for $i\neq j$, when $\partial_G(x)$ contains parallel edges.) The construction of the new graph $H$ is as follows. Attach six copies of $G-x$, say $G^0,G^1,\ldots, G^5$, whose vertices corresponding to $x_0,\ldots,x_5$ are $x_0^i,\ldots,x_5^i$ for $i=0,\ldots,5$. First, replace the vertex $v_0$ of $W$ with $G^0$ by putting the end $v_0$ of edge $v_0v_i$ in the position of $x_{i-1}^0$ for each $i=0,\ldots, 4$. Then, for each $j=1,\ldots,5$, replace the vertex $v_j$ of $W$ with $G^j$ by putting the end $v_j$ of edge $v_jv_0$ in the position of $x_0^j$, and putting the end $v_j$ of other edges in $\partial_W(v_j)$ matching to $x_1^j,x_2^j,x_3^j,x_4^j$, respectively.  The constructed new graph $H$ is depicted in Figure \ref{FIGW3flow}.

 It is routine to check that $H$ is $5$-edge-connected by the $5$-edge-connectivity of $W$ and copies of $G$.  Since statement (ii) holds, $H$ admits a mod $3$-orientation $D$. Contract all copies of $G-x$ to obtain a graph $W$ and consider the orientation $D$ restricted to $W$. By Lemma \ref{OB: W} (i), there exist a vertex $v_k$ of $W$, corresponding to the contraction of $G^k$ (for some $k\in\{1,\ldots,5\}$), such that $v_0v_k$ is the minor-edge at $v_k$. Now in $H$ contract all the vertices in $V(H)\setminus V(G^k)$ to become a new vertex $x$. Then this results a copy of $G$, consisting of a vertex $x$ and $G^k=G-x$. The orientation $D$ restricted to it provides a mod $3$-orientation $D_k$. Moreover, the edge $xx_0^k$ is a minor-edge at $x$ under $D_k$. If $D_k$ agrees with $D_x$ at $x$, then $D_k$ is a mod $3$-orientation extended from $D_x$, a contradiction. Otherwise, we reverse the orientation of all edges from $D_k$ to obtain another  mod $3$-orientation $D_k^*$. Now $D_k^*$ agrees with $D_x$ at $x$ since $xx_0^k$ is still the minor-edge at $x$ under $D_k^*$. This is a contradiction again, completing the proof of Theorem \ref{kochol}.
{\rule{3mm}{3mm}\par\medskip}

With a little more thought, one can  observe that in proving Theorem \ref{kochol}, if the graph $G$ is planar, then the constructed graph $H$ can be modified to planar as well, see similar construction in Figure \ref{FIGWZ3} below. (This is because the positions of $x_0,\ldots,x_4$ can be shifted cyclically in a planar embedding.)
Thus we obtain the following corollary for planar graphs. It suggests that Gr{\"o}tzsch's 3CT is exactly equivalent to its restriction to girth $5$ case, a fact maybe not known before.

\begin{corollary}
  The following are equivalent versions of Gr\"{o}tzsch's 3CT.

  (a) Every triangle-free planar graph is $3$-colorable.

  (b) Every planar graph of girth $5$ is $3$-colorable.
\end{corollary}
By applying arguments dual to the proof above (using dual graph of $W$ and dual constructions), one may also show that Gr\"{o}tzsch's 3CT is also equivalent to the statement that any pre-coloring of a $5$-cycle in a triangle-free planar graph can be extended to a $3$-coloring of the entire graph, a useful strengthening theorem proved by Thomassen \cite{Thom94}.



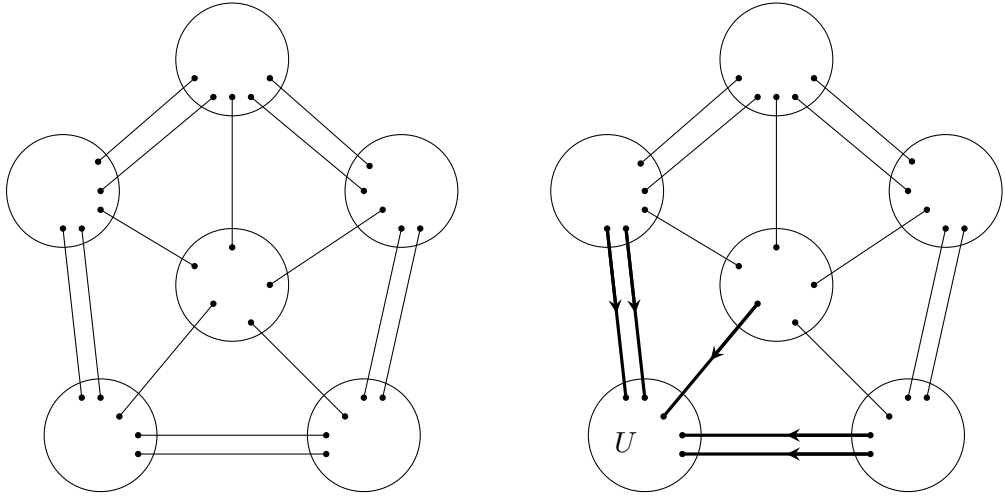
\begin{figure}[ht]
\centering
\begin{minipage}[ht]{0.4\textwidth}
  \begin{tikzpicture}[scale=0.5]
\tikzstyle{mynodestyle} = [draw,shape=circle,outer sep=0,inner sep=0.7,minimum size=2,fill=black]

\draw  (0,1) node (v4) {} ellipse (1.5 and 1.5);
\draw  (0,7) node (v5) {} ellipse (1.5 and 1.5);
\draw  (-4.5,3.5) node (v3) {} ellipse (1.5 and 1.5);
\draw  (4.5,3.5) ellipse (1.5 and 1.5);
\draw  (-3.5,-3) ellipse (1.5 and 1.5);
\draw  (3.5,-3) ellipse (1.5 and 1.5);

\node [mynodestyle] (v1) at (-4,-2) {};
\node [mynodestyle] (v19) at (-3.5,-2) {};
\node [mynodestyle] (v20) at (-3,-2.5) {};
\node [mynodestyle] (v22) at (-2.5,-3) {};
\node [mynodestyle] (v24) at (-2.5,-3.5) {};
\node [mynodestyle] (v6) at (2.5,-3.5) {};
\node [mynodestyle] (v25) at (2.5,-3) {};
\node [mynodestyle] (v23) at (3,-2.5) {};
\node [mynodestyle] (v26) at (3.5,-2) {};
\node [mynodestyle] (v28) at (4,-2) {};
\node [mynodestyle] (v8) at (5,2.5) {};
\node [mynodestyle] (v29) at (4.5,2.5) {};
\node [mynodestyle] (v27) at (4,3) {};
\node [mynodestyle] (v33) at (3.5,3.5) {};
\node [mynodestyle] (v31) at (3.6514,4.1651) {};

\node [mynodestyle] (v15) at (-1,6.5) {};

\node [mynodestyle] (v32) at (0,6) {};
\node [mynodestyle] (v30) at (0.5,6) {};
\node [mynodestyle] (v10) at (1,6.5) {};

\node [mynodestyle] (v18) at (-4.5,2.5) {};

\node [mynodestyle] (v14) at (-3.5,3.5) {};

\node [mynodestyle] (v12) at (-3.5651,4.2791) {};
\node [mynodestyle] (v2) at (-0.5,0.5) {};
\node [mynodestyle] (v7) at (0.5,0) {};
\node [mynodestyle] (v9) at (1,1) {};
\node [mynodestyle] (v11) at (0,2) {};
\node [mynodestyle] (v13) at (-1,1.5) {};
\node[mynodestyle](v35) at (-0.5,6) {};

\node [mynodestyle] (v17) at (-3.5,3) {};
\node [mynodestyle] (v21) at (-4,2.5) {};

\draw  (v18) edge (v1);

\draw  (v21) edge (v19);

\draw  (v2) edge (v20);
\draw  (v13) edge (v17);

\draw  (v12) edge (v15);

\draw  (v35) edge (v14);

\draw  (v32) edge (v11);
\draw  (v30) edge (v33);
\draw  (v10) edge (v31);
\draw  (v9) edge (v27);
\draw  (v29) edge (v26);
\draw  (v8) edge (v28);
\draw  (v7) edge (v23);
\draw  (v22) edge (v25);
\draw  (v24) edge (v6);
\end{tikzpicture}
\end{minipage}
\hspace{0.5cm}
\begin{minipage}[ht]{0.4\textwidth}
\begin{tikzpicture}[scale=0.5]
\tikzstyle{mynodestyle} = [draw,shape=circle,outer sep=0,inner sep=0.7,minimum size=2,fill=black]

\draw  (0,1) node (v4) {} ellipse (1.5 and 1.5);
\draw  (0,7) node (v5) {} ellipse (1.5 and 1.5);
\draw  (-4.5,3.5) node (v3) {} ellipse (1.5 and 1.5);
\draw  (4.5,3.5) ellipse (1.5 and 1.5);
\draw  (-3.5,-3) ellipse (1.5 and 1.5);
\draw  (3.5,-3) ellipse (1.5 and 1.5);

\node [mynodestyle] (v1) at (-4,-2) {};
\node [mynodestyle] (v19) at (-3.5,-2) {};
\node [mynodestyle] (v20) at (-3,-2.5) {};
\node [mynodestyle] (v22) at (-2.5,-3) {};
\node [mynodestyle] (v24) at (-2.5,-3.5) {};
\node [mynodestyle] (v6) at (2.5,-3.5) {};
\node [mynodestyle] (v25) at (2.5,-3) {};
\node [mynodestyle] (v23) at (3,-2.5) {};
\node [mynodestyle] (v26) at (3.5,-2) {};
\node [mynodestyle] (v28) at (4,-2) {};
\node [mynodestyle] (v8) at (5,2.5) {};
\node [mynodestyle] (v29) at (4.5,2.5) {};
\node [mynodestyle] (v27) at (4,3) {};
\node [mynodestyle] (v33) at (3.5,3.5) {};
\node [mynodestyle] (v31) at (3.6048,4.286) {};

\node [mynodestyle] (v15) at (-1,6.5) {};

\node [mynodestyle] (v32) at (0,6) {};
\node [mynodestyle] (v30) at (0.5,6) {};
\node [mynodestyle] (v10) at (1,6.5) {};

\node [mynodestyle] (v18) at (-4.5,2.5) {};

\node [mynodestyle] (v14) at (-3.5,3.5) {};

\node [mynodestyle] (v12) at (-3.6087,4.2302) {};
\node [mynodestyle] (v2) at (-0.5,0.5) {};
\node [mynodestyle] (v7) at (0.5,0) {};
\node [mynodestyle] (v9) at (1,1) {};
\node [mynodestyle] (v11) at (0,2) {};
\node [mynodestyle] (v13) at (-1,1.5) {};
\node[mynodestyle](v35) at (-0.5,6) {};

\node [mynodestyle] (v17) at (-3.5,3) {};
\node [mynodestyle] (v21) at (-4,2.5) {};

\draw [very thick] (v18) edge (v1);

\draw [very thick] (v21) edge (v19);

\draw [very thick] (v2) edge (v20);
\draw  (v13) edge (v17);

\draw  (v12) edge (v15);

\draw  (v35) edge (v14);

\draw  (v32) edge (v11);
\draw  (v30) edge (v33);
\draw  (v10) edge (v31);
\draw  (v9) edge (v27);
\draw  (v29) edge (v26);
\draw  (v8) edge (v28);
\draw  (v7) edge (v23);
\draw [very thick] (v22) edge (v25);
\draw [very thick] (v24) edge (v6);
\node (v16) at (-4.2,0) {};
\node (v34) at (-3.8,0.5) {};
\node (v36) at (-2,-1) {};
\node (v37) at (0,-3) {};
\node (v38) at (0,-3.5) {};
\draw [very thick,-stealth] (v18) edge (-4.25,0.25);
\draw [very thick, -stealth] (v21) edge (-3.75,0.25);
\draw [very thick,-stealth] (v2) edge (-1.75,-1);
\draw [very thick,-stealth] (v25) edge (v37);
\draw [very thick,-stealth] (v6) edge (v38);
\node [label={$U$}] at (-4,-4) {};

\end{tikzpicture}
\end{minipage}
\caption{The constructed graph for proving Theorem \ref{THM: TFAE} and Theorem \ref{THM:pla}(ii)(iii).}\label{FIGWZ3}
\end{figure}

~

Now we prove Theorem \ref{THM: TFAE} using similar constructions, but employing Lemma \ref{OB: W}(ii) instead. The argument presented here is a slight modification of that in the author's Ph.D dissertation\cite{JA18}.

{\bf Proof of Theorem \ref{THM: TFAE}:}
The relations of some of those statements have been investigated in  \cite{HLL18}. The proofs of ``(b-i)$\Leftrightarrow$(b-ii)'' and  ``(b-ii)$\Rightarrow$(c)$\Rightarrow$(a)'' have been presented in \cite{HLL18,JA18}. 
 Clearly, we also have ``(d)$\Rightarrow$(a)''. We shall complete the proof of Theorem \ref{THM: TFAE} by showing ``(b-i)$\Rightarrow$(d)'' and ``(a)$\Rightarrow$(b-i)'' below.\\

{\em Proof of ``(b-i)$\Rightarrow$(d)'':} Let $G$ be a $4$-edge-connected graph with at most five $4$-edge-cuts. Denote $A_1, A_2, \dots, A_t$ to be all distinct $4$-critical-sets $A$ such that $\partial(A)$ is a $4$-critical-cut. Then $t\le 5$ by Observation \ref{OB: kedgeconnected}. The conclusion is clear if $t=0$. We may assume $1\le t\le 5$. Construct a new graph $G'$ from $G$ by adding a new vertex $z$, connecting $z$ and $A_1$ with $6-t$ new edges,  and connecting $z$ and $A_i$ with a new edge for each $i=2, \ldots, t$. Then $d_{G'}(z)=5$ and $G'$ is $5$-edge-connected by Observation \ref{OB: kedgeconnected} (iv). By the validity of Theorem \ref{THM: TFAE} (b-i), $G'$ is $\Z_3$-extendable at $z$. Then it follows from Lemma \ref{extendingiff} that $G=G'-z$ is $\Z_3$-connected. This proves ``(b-i)$\Rightarrow$(d)''.\\

{\em Proof of ``(a)$\Rightarrow$(b-i)'':} Suppose to the contrary that $G$ is a $5$-edge-connected graph which
 is not $\Z_3$-extendable at a given $5$-vertex $z$. By Proposition \ref{extendingiff}, $G-z$ is not $\Z_3$-connected, and thus $G-z$ has no $\beta_1$-orientation for some boundary function $\beta_1$ of $G-z$. Denote $\partial(z)=\{zu_1,zu_2,\ldots,zu_5\}$. (Note that $u_i$, $u_j$ may represent the same vertex for distinct $i$ and $j$ when $\partial(z)$ contains parallel edges.) We orient the edge $zu_i$ from $z$ to $u_i$ for each $i=1,\ldots, 5$ to obtain a pre-orientation $D_{z}$. Let $\beta$ be a boundary function of $G$ such that $\beta(z)=2$ and $\beta(x)=\beta_1(x)-\alpha(x)$ in $\mathbb{Z}_3$ for any $x\in V(G)\setminus\{z\}$, where $\alpha(x)$ is the number of directed edges from $z$ to $x$. (In particular, $\beta(x)=\beta_1(x)$ in $\mathbb{Z}_3$ for any $x\in V(G)-\cup_{i=1}^5\{u_i\}\cup\{z\}$.) Clearly, $\beta\in Z(G, \Z_3)$ and
 \begin{equation}\label{EQ:Dznotextend}
  \text{$D_{z}$ cannot be extended to a $\beta$-orientation of $G$.}
\end{equation}

 Now, we replace each vertex of the graph $W$ (see Figure \ref{FIGW}) with a copy of $G-z$, where each $u_i$ is connected with an edge of $W$ (see Figure \ref{FIGWZ3}). Let $H$ be the resulting graph. Define a boundary function $\beta^*$ of $H$ such that  $\beta^*$ is consistent with $\beta$ in each copy of $G-z$. Note that $\beta^*$ is indeed a boundary function of $H$ as $\sum_{v\in V(H)}\beta^*(v)=6\sum_{v\in V(G-z_0)}\beta(v)\equiv 0\pmod 3$. Since $H$ is $5$-edge-connected, we have a $\beta^*$-orientation $D^*$ of $H$ by the validity of Theorem \ref{THM: TFAE} (a). Under the orientation $D^*$, we consider the oriented graph $W$ obtained from $H$ by contracting all the copies of $G-z$. By Lemma \ref{OB: W}(ii), there exists a vertex with indegree $5$. We uncontract this vertex and denote its corresponding vertex set of $H$ by $U$. Then $H/U^c$ is isomorphic to $G$, where the contracted vertex $y$ plays the same role as $z$. Furthermore, the orientation $D^*$ restricted to $H/U^c$ gives a $\beta$-orientation of $H/U^c$ since all the edges incident with $y$ are directed out of $y$. This contradicts to (\ref{EQ:Dznotextend}) that $D_{z}$ cannot be extended to a $\beta$-orientation of $G$. The proof is completed.
{\rule{3mm}{3mm}\par\medskip}

~

Now we prove Theorem \ref{THM:pla} using similar arguments as in the proof of Theorem \ref{THM: TFAE}.

{\bf Proof of Theorem \ref{THM:pla}:} The proof of ``(i)$\Rightarrow$(ii)'' is the same as the proof of Theorem \ref{THM: TFAE} ``(a)$\Rightarrow$(b-i)'' above. Notice that when $G$ is planar, the new constructed graph $H$ from $W$ and copies of $G-x$ is also planar, and hence  ``(i)$\Rightarrow$(ii)'' holds.   The proof of ``(ii)$\Rightarrow$(iii)'' is also straightforward by employing Lemma \ref{extendingiff}, similar as proving Theorem \ref{THM: TFAE} ``(b-ii)$\Rightarrow$(c)'' in \cite{HLL18}. If there exists a $\Z_3$-reduced graph with minimal degree at least $5$, we choose a vertex set $S$ such that $\partial(S)$ is a $4$-critical-set. Then $|S|\ge 2$, and contract $S^c$ to obtain a graph $G_1=G/S^c$, where $x$ is the vertex set $S^c$ contracted into. Add $5-|\partial_G(S)|$ edge between $x$ and $S$ in $G_1$ to result a new planar graph $G_2$. Hence $G_2$ is $5$-edge-connected. By (ii), $G_2$ is $\Z_3$-extendable at $x$, which shows that $G[S]=G_2-x$ is $\Z_3$-connected by Lemma  \ref{extendingiff}, a contradiction to the fact that $G$ is $\Z_3$-reduced.

Now we prove ``(ii)$\Rightarrow$(iv)'' with similar arguments. Let $G$ be a $5$-edge-connected graph embedded on the plane such that the only crossing is between $x_1x_2$ and $y_1y_2$. We delete  edges $x_1x_2,y_1y_2$ and add a new vertex $z$ with edges $zx_1,zx_2, zy_1,zy_2,zy_2$. Let $G'$ be the resulting graph. Then $G'$ is a $5$-edge-connected planar graph with a $5$-vertex $z$. By (ii), $G'$ is $\Z_3$-extendable at $z$, and hence $G'-z=G-x_1x_2-y_1y_2$ is $\Z_3$-connected by Lemma  \ref{extendingiff}. Thus $G$ is $\Z_3$-connected. This completes the proof of Theorem \ref{THM:pla}.
{\rule{3mm}{3mm}\par\medskip}

One may wonder whether the proof of Theorem \ref{THM:pla} extends to the ``doublecross graphs'', graphs can be drawn in the
plane with two crossings incident with the infinite region.  We are unable to reduce it to planar case as in Theorem \ref{THM:pla}. Similar phenomenon happens for Four Color Theorem(4CT) of planar graphs.  Jaeger \cite{Jaeg80} proved that every bridgeless cubic graph with at most one crossing has a nowhere-zero $4$-flow (equivalently, is $3$-edge-colorable), which is reduced to the planar case, an equivalent version of 4CT, that every bridgeless cubic planar graph has a nowhere-zero $4$-flow. However, for doublecross cubic graphs, Edwards, Sanders, Seymour and Thomas \cite{ESST16} employed the whole arguments of 4CT proofs (and many more works) to accomplish their proof that every bridgeless doublecross cubic graph has a nowhere-zero $4$-flow.

\section{Graphs with Few Small Critical-cuts}
We prove Theorem \ref{THM: ETFH} (d-i)(d-ii)  in this section. Evidently, Theorem \ref{THM: ETFH} (d-ii) is easily derived by Theorem \ref{THM: ETFH} (b-i) and Observation \ref{OB: kedgeconnected}. However, Theorem \ref{THM: ETFH} (d-i) seems not to be deduced from the current version of Theorem \ref{THM: ETFH} (b-i). We shall apply the full version of the flow extension theorem of Lov\'{a}sz et al. \cite{LTWZ13}.

Let $G$ be a graph and $\beta$ a boundary function. For a vertex set $A\subset V(G)$, denote its boundary $\beta(A)\equiv \sum_{x\in A}\beta(x)\pmod3$. Define an integer valued mapping
$\tau : 2^{V(G)}  \mapsto \{0,\pm 1, \pm 2, \pm 3\}$  such that,  for each vertex set $A\subset V(G)$, $\tau (A) \equiv d(A) \pmod 2$ and $\tau (A) \equiv \beta(A) \pmod 3$.

\begin{theorem}\label{partialextending}
{\em(Lov\'{a}sz et al. \cite{LTWZ13})}
Let $G$ be a graph, $\beta\in Z(G,\Z_3)$ and $z \in V(G)$.
Let $D_{z}$ be a pre-orientation of $\partial_G({z})$. Assume that
\\
(i) $|V(G)|\ge 3$,
\\
(ii) $d(z)\le 4 + |\tau(z)|$ and $d_{D_z}^+(z)-d_{D_z}^-(z)\equiv \beta(z) \pmod 3$, and
\\
(iii) $d(A)\ge 4+ |\tau(A)|$ for each nonempty $A\subseteq V(G)-\{z_0\}$ with $|V(G)- A|\ge 2$.
\\
Then $D_{z}$ can be extended to a $\beta$-orientation of the entire graph $G$.
\end{theorem}

Now we are ready to prove Theorem \ref{THM: ETFH} (d-i)(d-ii) using Theorem \ref{partialextending}.

{\bf Proof of Theorem \ref{THM: ETFH}.} {\em Proof of (d-i):} Let $G$ be a $4$-edge-connected graph with at most five $4$-cuts and without $5$-cuts. Let $\beta\in Z(G,\Z_3)$ be a boundary function of $G$. We are going to show that $G$ has a $\beta$-orientation.   Similar to the previous section, we denote $A_1, A_2, \dots, A_t$ to be all distinct  $4$-critical-sets of $G$. Note that $t\le 5$ by Observation \ref{OB: kedgeconnected}. Construct a new graph $G'$ from $G$ by adding a new vertex $z$, and for each $i=1, \ldots, t$, adding a new edge between $z$ and $A_i$, say $zv_i$ (where $v_i\in A_i$). We pre-orient the edges in $\partial_{G'}(z)$ and modify the boundary appropriately to become a new boundary $\beta'$ of $G'$ such that $d_{G'}(A_i)=4+|\tau'(A_i)|$ for each $i=1, \ldots, t$, where $\tau'$ denotes the $\tau$-function corresponding to boundary $\beta'$ in $G'$. Specifically, we orient the edge $zv_i$ from $z$ to $v_i$ if $\tau(A_i)=0$ or $2$, and orient $zv_i$ from $v_i$ to $z$ otherwise(i.e. $\tau(A_i)=-2$).  Define the boundary  $\beta'$ of $G'$ as follows. For any $x\in V(G')\setminus\{v_1,\ldots,v_t\}$, define $\beta'(x)=\beta(x)$; for each $i=1, \ldots, t$, define $\beta'(v_i)=\beta(v_i)+1$ if $zv_i$ is oriented from $v_i$ to $z$, and $\beta'(v_i)=\beta(v_i)-1$ otherwise. Now, it is easy to see that   $d_{G'}(A_i)=4+|\tau'(A_i)|$ for each $i=1, \ldots, t$, and that Theorem \ref{partialextending} is applied for $G'$ by checking conditions (i)(ii)(iii). That is, we have $d_{G'}(z)\le 4 + |\tau'(z)|$ since $d_{G'}(z)\le 5$ and by parity, and this verifies condition (ii) of Theorem \ref{partialextending}.  Let $A$ be a nonempty subset of $V(G')-\{z\}$ with $|V(G')- A|\ge 2$.  If $d_{G'}(A)\ge 6$, then we have $d_{G'}(A)\ge 4+ |\tau'(A)|$ by parity. Otherwise, we have $A=A_i$ for some $i$, and so $d_{G'}(A)= 4+ |\tau'(A)|$. Hence condition (iii) of Theorem \ref{partialextending} holds. By Theorem \ref{partialextending}, the pre-orientation can be extended to a $\beta'$-orientation $D'$ of $G'$. Notice that $D'$ restricted to $G$ provides a $\beta$-orientation of $G$. This proves (d-i).

{\em Proof of (d-ii):} The proof of (d-ii) is analogous to the proof of Theorem \ref{THM: TFAE} ``(b-i)$\Rightarrow$(d)''. We add a new vertex $z$ to connect  each $5$-critical-set to obtain a new graph $G'$ such that $d_{G'}(z)=7$. Then $G=G'-z$ is $\Z_3$-connected by Theorem \ref{THM: ETFH} (b-i) and Lemma \ref{extendingiff}. This completes the proof.
{\rule{3mm}{3mm}\par\medskip}

Note that, by Observation \ref{OB: kedgeconnected} the proof above is still valid for graphs with many $5$-cuts but only at most seven $5$-critical-cuts,  with essentially the same proof.
\begin{corollary}
 Every $5$-edge-connected graph with at most seven $5$-critical-cuts is $\Z_3$-connected.
\end{corollary}

%
%
%
%
%
%


\end{document}